\pdfoutput=1
\documentclass[11pt]{amsart}
\usepackage{mathrsfs, amsmath, amssymb,graphics,booktabs}
\usepackage[all]{xy}
\usepackage{verbatim}
\usepackage{color}
\usepackage{graphicx}
\usepackage[title]{appendix}

\newtheorem{theorem}{Theorem}
\newtheorem{lemma}[theorem]{Lemma}
\newtheorem{proposition}[theorem]{Proposition}

\newtheorem{corollary}[theorem]{Corollary}

\theoremstyle{definition}
\newtheorem{definition}[theorem]{Definition}

\theoremstyle{remark}

\def\Mod{{\rm Mod}}
\def\Homeo{{\rm Homeo}}
\def\Sp{{\rm Sp}}

\def\Out{{\rm Out}}
\def\Aut{{\rm Aut}}
\def\GL{{\rm GL}}
\def\SL{{\rm SL}}
\def\Ker{{\rm Ker}}
\def\orb{{\rm orb}}

\def\Hom{{\rm Hom}}
\def\Cay{{\rm Cay}}

\def\modd{{\rm mod}}

\numberwithin{equation}{section}

\newcommand{\RNum}[1]{\uppercase\expandafter{\romannumeral #1\relax}}

\begin{document}

\title{On profinite rigidity of 4-dimensional Seifert manifolds}

\author{Jiming Ma}
\address{School of Mathematical Sciences, Fudan University, Shanghai, 200433, P. R. China}
\email{majiming@fudan.edu.cn}

\author{Zixi Wang*}
\address{School of Mathematical Sciences, Zhejiang Normal University, Zhejiang, 321004, P. R. China}

\email{zxwang22@zjnu.edu.cn}



\date{Jun 12, 2022}

\thanks{Jiming Ma was partially supported by NSFC 12171092.}

\begin{abstract}
There are many results showing the connection and phenomenon between some low-dimensional manifolds with the profinite completions of their fundamental groups. We focus on some Seifert 4-manifolds about the extent of their profinite completion to detect one, giving classification of monodromies and conditions for them to be profinitely rigid.

\end{abstract}
%%%%%%%%%%%%%%%%%%%%%%%%%%%%%%%%%%%%%%%%%%%%%%%%%%%%%%%%%%%%%%%%%%%%%%%%

\maketitle

\section{introduction}

It has been long focused that distinguishing finitely generated residually finite group $G$ by the collection of its finite quotients, or equivalently, by the \textit{profinite completion} $\widehat{G}$. We care about this problem especially for the fundamental groups of aspherical manifolds corresponding to their finite coverings.
Let $\mathscr{C}$ denote a certain collection of residually finite groups. We say a group $G\in\mathscr{C}$ is \textit{profinitely rigid in $\mathscr{C}$} if $G$ could be distinguished by $\widehat{G}$ in $\mathscr{C}$ up to isomorphism. Otherwise $G$ is called \textit{profinitely flexible} when the rigidity fails. When $\mathscr{C}$ is the collection of all finitely generated residually finite groups we say the profinite rigidity is in \textit{absolute sense}. Considering about free group $F_n$ of rank $n\ge2$, it's still open that whether $F_n$ is absolutely profinitely rigid \cite{Noskov:1979,Reid:2018}.

In 2018 ICM, Reid introduced the profinite rigidity problem under low-dimensional topological condition \cite{Reid:2018}.  Bridson, Conder and Reid showed that Fuchsian groups are profinitely rigid among all the lattices of connected Lie groups \cite{BCR:2015}.  For a 3-dimensional manifold with one of Thurston's geometry, its geometric type could be distinguished by the profinite completion of its fundamental group proved by Wilton and Zalesskii \cite{Wil:2017}. Bridson, McReynolds, Reid and Spitler showed that there exist arithmetic lattices in $PSL(2,\mathbb{C})$ which are profinitely rigid in the absolute sense \cite{BMRS:2018}. Jaikin-Zapirain proved that profinite completion detects fiberness of compact orientable aspherical 3-manifolds \cite{J:2020}. Bridson and Reid \cite{BrRe:2015}, Boileau and Friedl \cite{BoFr:2015}  independently proved that the figure-eight knot group could be distinguished by its fundamental group among all the knot groups. However there are examples of 3-manifolds which are not profinitely rigid. Funar \cite{Funar} gave infinite pairs of non-homeomorphic $\mathbb{S}ol^3$ manifolds with isomorphic profinite completions based on the work of Stebe \cite{Stabe:1972}. Hempel \cite{Hem:2014} showed that there are non-homeomorphic closed Seifert fibred spaces of geometry $\mathbb{H}^2\times\mathbb{E}$ with isomorphic profinite completions. Wilkes \cite{Wilk:2017} proved that the profinite rigidity holds in the 3-dimensional closed orientable Seifert fibred manifolds except for the case exhibited by Hempel.

Considering the 4-dimensional Thurston's geometries, Stover \cite{Sto:2019} showed the profinite flexibility of $\mathbb{H}_{\mathbb{C}}^2$-manifolds. Piwek, Popovi\'{c} and Wilkes proved that two 4-dimensional $\mathbb{E}^4$-orbifolds are profinitely rigid among themselves seeing their fundamental groups as 4-dimensional crystallographic groups \cite{PPW:2019}. We proved that most 4-dimensional geometries in the sense of Thurston could be distinguished by profinite completions except for $\mathbb{H}^4$ and $\mathbb{H}_{\mathbb{C}}^2$ \cite{MaWa:2022}. Ue connected eight types of these geometries with 4-dimensional Seifert manifolds and gave the concrete connection how the fibering types form certain geometric structure\cite{Ue:1990}. Then we focus on distinguishing 4-dimensional Seifert manifolds with their profinite completions. The first procedure is to modify the fiber presentation in a more simplified way.

\begin{theorem}\label{monodromy}
	Suppose $M$ is a 4-dimensional $\mathbb{H}^2\times\mathbb{E}^2$ manifold with non-trivial monodromy. Then there is always a set of generators of $\pi_1(M)$ such that the corresponding monodromies are all trivial except for one matrix.
\end{theorem}

When the monodromies are all trivial, we proved the profinite rigidity of $\widetilde{\SL_2}\times\mathbb{E}$ and gave the condition for  $\mathbb{H}^2\times\mathbb{E}^2$-manifolds to be profinitely isomorphic.

\begin{theorem}\label{Seifert}
	The closed orientable Seifert 4-manifolds with  $\widetilde{\SL_2}\times\mathbb{E}$ are profinitely rigid. For two closed orientable Seifert 4-manifolds $M$ and $N$ with $\mathbb{H}^2\times\mathbb{E}^2$-geometry of trivial monodromy, then $\widehat{{\pi_1(M)}}\cong\widehat{{\pi_1(N)}}$ if and only if there exists a profinite integer $k\in \mathbb{Z}$ which is related to the profinite isomorphism such that the Seifert invariants of $M$ and $N$ are $(m_i,a_i,b_i)$ and $(m_i,ka_i,kb_i)$ respectively.
\end{theorem}

It should be remarked that our result in Theorem \ref{Seifert} is overlapped with Piwek's work \cite{Piw} when the orbifold contains cone points more than $1$. This is an original work considering the tight gap of posting time.

\section{Profinite completions}

Now we introduce some elementary results of profinite groups based on the book of Ribes and Zalesskii \cite{RiZa:2004}.
\begin{definition}
	For a partially ordered set $I$, an \textit{inverse system} is a family of sets $\{X_i\}_{i\in I}$, and a family of maps $\phi_{ij}:X_i\rightarrow X_j$ whenever $i\leq j$ such that:
	
	(1) $\phi_{ii}=id_{X_i}$,
	
	(2) $\phi_{jk}\phi_{ij}=\phi_{ik}$ if $i\leq j\leq k$.
	
	Denoting the inverse system as $(X_i,\phi_{ij},I)$, its \textit{inverse limit} is defined as
	\begin{center}
		$\varprojlim X_i=\{(x_i)\in\prod_{i\in I} X_i|\ \phi_{ij}(x_i)=x_j \ if \ i\leq j\}$.
	\end{center}
	
\end{definition}

\begin{definition}
	Given an inverse system $(X_i,\phi_{ij},I)$ such that $X_i$ are finite groups and $\phi_{ij}$ are group homomorphisms, the resulting inverse limit $\varprojlim X_i$ is called a \textit{profinite group}. For a discrete group $G$, the collection $\mathcal{N}$ of its finite index normal subgroups could form an inverse system by declaring that $N_i\leq N_j$ whenever  $N_i,N_j\in\mathcal{N}$ and $N_i\subseteq N_j$. So the group homomorphisms are natural epimorphisms $\phi_{ij}:G/N_i\rightarrow G/N_j$. We call the inverse limit of the inverse system $(G/N_i,\phi_{ij},\mathcal{N})$  the \textit{profinite completion} of $G$ and denote it as $\widehat{G}$.
	
\end{definition}

There is another important property to introduce:

\begin{definition}
	A group $G$ is \textit{residually finite} if for every non-trivial $g\in G$, there exists some finite index normal subgroup $N\triangleleft G$ such that $g\notin N$.
\end{definition}

A group $G$ is residually finite means that $\bigcap_{N_i\in \mathcal{N}} N_i=1$, and $G$ is residually finite if and only if $\iota:G\rightarrow \widehat{G}$ is injective where $\iota$ is induced by the natural map $G\rightarrow G\slash N_i$ for every $N_i\in\mathcal{N}$.

Considering the discrete topology of the quotients $G/N_i$, the topology on $\widehat{G}$ could be induced as the sub-topology of the product topology on $\prod_{N_i\in\mathcal{N}} G/N_i$. Then the profinite completion $\widehat{G}$ is a compact, Hausdorff and totally disconnected group. An open subgroup $H$ of $\widehat{G}$ in the profinite topology means that $H$ is closed of finite index in the sense of topological groups, and $H$ is closed if and only if it is the intersection of all open subgroups of $\widehat{G}$ containing $H$.  On the other hand, when considering the profinite topology on $\widehat{G}$,  the connection between the discrete group and its profinite completion has been made:

\begin{proposition}\label{profinite1-1}\cite{RiZa:2004}
	If $G$ is a finitely generated residually finite group, then there is a one-to-one correspondence between the set $\mathcal{X}$ of subgroups of $G$ with finite index and the set $\mathcal{Y}$ of all open subgroups of $\widehat{G}$. Identifying $G$ with its image in $\widehat{G}$, the correspondence is given by:
	
	(1) for $H\in \mathcal{X}$, $H\longmapsto \overline{H}$,
	
	(2) for $Y\in \mathcal{Y}$, $Y\longmapsto Y\cap G$.
	
	Here $\overline{H}$ means the closure of $H$ under the profinite topology of $\widehat{G}$. Further more, if $K,H\in \mathcal{X}$ and $K\le H$ then $[H:K]=[\overline{H}:\overline{K}]$; $H$ is normal in $G$ if and only if $\overline{H}$ is normal in $\widehat{G}$; and finally, for $H\in \mathcal{X}$, $\overline{H}\cong\widehat{H}$.
\end{proposition}

When we consider the finite quotients of $\widehat{G}$ for which the quotient maps are continuous, it is clear that $G$ and $\widehat{G}$ have the same set of finite quotients. Let $\mathcal{C} (G)$ denote the collection of finite quotients of $G$, then we have the following theorem proved in \cite{RiZa:2004}:

\begin{theorem}
	Let $G_1$ and $G_2$ be two finitely generated abstract groups, then $\widehat{G_1}\cong \widehat{G_2}$ if and only if $\mathcal{C} (G_1)=\mathcal{C} (G_2)$.
\end{theorem}

The homology and cohomology of one's profinite completion maintain many features from the discrete group. Serre defined goodness in \cite{Se:2013}:

\begin{definition}
	A finitely generated group $G$ is \textit{good} if for any finite $G$-module $A$, the natural homomorphism of cohomology groups $H^{n}(\widehat{G};A) \longrightarrow H^{n}(G;A) $ induced by $\iota :G\longrightarrow \widehat{G}$ is an isomorphism for any dimension $n$.
\end{definition}

There are some examples for illustration. It is shown that the finitely generated Fuchsian groups are good \cite{GJZZ:2008}, and the extension of a good group by another good group is itself good. It is one application of goodness as follows \cite{GJZZ:2008}:

\begin{lemma}\label{extension}
	Let G be a residually finite, good group and $\phi:H\rightarrow G$ be a surjective homomorphism with residually finite and finitely generated kernel $K$. Then $H$ is residually finite.
\end{lemma}

This lemma provides an efficient way to build a large collection of residually finite groups which are extensions of residually finite groups and gives what we need for Seifert 4-manifolds.

\begin{corollary}
	Let $M$ be a Seifert 4-manifold over hyperbolic 2-orbifold. Then $\pi_1(M)$ is residually finite.
\end{corollary}

\section{Seifert fibred 4-manifolds}

What people are familiar with is the 3-dimensional Seifert fibred spaces which are circle bundle over 2-orbifolds. They are classified and associated to 3-dimensional Thurston's geometries by the orbifold characteristic number and Euler number of bundles \cite{Scott:1983}. The idea of \textit{Seifert fibred 4-manifold} $M$ is similar which means the total space of a bundle $\pi:M\rightarrow B$ over a closed $2$-orbifold $B$ such that the general fiber is $2$-torus $T^2$ or Klein bottle $K$ \cite{Ue:1990}. And Ue also characterizes Seifert 4-manifolds and illustrates the connection of 4-dimensional geometries in the sense of Thurston. While the fiber $\pi_1(T^2)$ admits non-trivial action of $\pi_1^{orb}(B)$, which brings more flexibility when comparing to $3$-dim case as the difference between outer automorphism group of $\pi_1(T^2)$ and of $\pi_1(S^1)$. The change of fiber also leads to more procedure in exact sequence and cohomology. The explicit characterize of Seifert 4-manifold by Ue gives us valid details for further work. 

The orbifold $B$ could be described by its underlying space $|B|$ and singular points including cone points, reflectors and corner reflectors associated to different type of discrete actions. Locally, the singular points $p\in B$ has a sufficiently small neighborhood $D=\mathbb{D}^2/G$ which is the quotient space of unit disk by a discrete subgroup  $G<O(2)$ fixing $p$. Then $\pi^{-1}(D)$ is identified with $T^2\times D^2\slash G$ such that $G$ acts on $T^2\times D^2$ freely. The $G$-actions on $T^2\times D^2$ and $D^2$ are compatible with the natural projection $T^2\times D^2\to D^2$. We identify the points of $T^2\times D^2$ as $(x,y,z)$ such that $(x,y)\in\mathbb{R}^2\slash \mathbb{Z}^2$ and $z\in\mathbb{C}$, $|z|\leq 1$.
 
If $G$ is trivial, we say the point $p\in B$  is a \textit{non-singular point}, and the preimage of $D$ in the total space is just $T^2\times D^2$.

If $G=\mathbb{Z}_{m}$ for $m\geq2$ generated by action $\rho$ on $T^2\times D^2$ such that $\rho(x,y,z)=(x-\frac{a}{m},\ y-\frac{b}{m},\ e^{\frac{2\pi i}{m}}z)$ with $g.c.d.(m,a,b)=1$. Then we say $p$ is a \textit{cone point} of $B$ and the fiber around $p$ is called a \textit{multiple torus of type $(m,a,b)$}.

If $G=\mathbb{Z}_{2}$ where the generator $\iota$ acts on $T^2\times D^2$ by $\iota(x,y,z)=(x+\frac{1}{2},\ -y,\ \overline{z})$. In this case, the underlying space $|B|$ is a surface with boundary such that each connected component of boundary forms a reflector circle of $B$. Then $p$ lies on one reflect circle and $\pi^{-1}(D)$ is a twisted $D^2$-bundle over the Klein bottle $K$. In this case, we still call $B$ as a closed orbifold even its underlying space is not a closed surface.

If $G=D_{2m}=\left\langle \iota,\ \rho |\ \iota^2=\rho^m=1, \iota\rho\iota^{-1}=\rho^{-1}\right\rangle$  is the dihedral group  where the generators $\iota$ and $\rho$ act on $T^2\times D^2$ by $\iota(x,y,z)=(x+\frac{1}{2},\ -y,\ \overline{z})$ and $\rho(x,y,z)=(x,\ y-\frac{b}{m},\ e^{\frac{2\pi i}{m}}z)$ with $g.c.d.(m,b)=1$. Then we say $p$ is a \textit{corner reflector} of angle $\pi\slash m$ and the fiber over $p$ is called a \textit{multiple Klein bottle of type $(m,0,b)$}. It means that the base orbifold has reflector circles such that each circle contains finitely many corner reflectors.  From now on, we don't deal with the Seifert 4-manifolds with reflector circles on purpose, for each of them has a canonical double cover which is a Seifert 4-manifold without reflector circles.

We still need describe the fibration globally which requires some invariants similar to 3-dimensional case. Suppose the base orbifold $B$ has no reflector circles, then we use the 4-dimensional Seifert invariants defined by Ue \cite{Ue:1990}:

(1) the \textit{monodromy matrices} $A_i,B_i\in SL(2,\mathbb{Z})$ satisfying $\prod_{i=1}^{g}[A_i,B_i]=I$ along the collection of standard generators $s_i,t_i\ (i=1,2,...,g)$ of $\pi_1(|B|)$, where the underlying space $|B|$ of $B$ is an orientable surface of genus $g$.

(1') the \textit{monodromy matrices} $A'_i\in GL_2(\mathbb{Z})$ along the collection of standard generators $v_i\ (i=1,2,...,g')$ of $\pi_1(|B'|)$, where the underlying space $|B'|$ of $B'$ is a non-orientable surface of genus $g'$.

(2) the tuple $(m_i,a_i,b_i)$ to describe the fiber type over cone points $p_i\ (i=1,...,t)$ of the base orbifold.

(3) the \textit{obstruction} $(a',b')\in\mathbb{Z}^2$. Let $q_i$ be the lift of the meridian circle centered at a non-singular point $p_i$, then $(a',b')$ is the obstruction to extend $\cup q_i $ to the cross section in $\pi^{-1}(B-\cup$(the disk neighbourhood of $p_i))$.

(4) When all the monodromies are trivial, we could define the euler number as $e=(a'+\Sigma a_i/m_i,\ b'+\Sigma b_i/m_i)\in\mathbb{Q}^2$  which is similar to the definition of rational euler number in 3-dimensional Seifert fibred space. The details are in \cite{Ue:1990}.

Now the fundamental group of $M$ could be represented as extension of orbifold fundamental group of $B$ by $\mathbb{Z}^2$ by the 4-dimensional Seifert invariants. Explicitly, the fundamental group of $M$ is
\begin{align*}
	\pi_1(M)=\langle x_1,\cdots,x_r,u_1,\cdots,v_g,l,h\  |\ u_j(l,h)u_j^{-1}=(l,h)A_j&,v_j(l,h)v_j^{-1}=(l,h)B_j\\
	x_i^{m_i}l^{a_i}h^{b_i}=1,\ x_1\cdots x_r\prod[u_j,v_j]=l^a h^b,[x_i,l]=[x_i,h]=[l,h]=1&\rangle
\end{align*}
where $l,h$ are the generators of the general fiber $T^2$ and $u_1,\cdots,v_g$ are the lift of generators of genus of $B$. It is apparent that the representation is not unique since there is a certain transformation between two sets of Seifert invariants of one Seifert 4-manifold.

By classifying closed orientable 4-dimensional Seifert fibred manifolds using the 4-dimensional Seifert invariants, Ue connected these manifolds to Thurston's geometries:

\begin{theorem}\cite{Ue:1990}
	Let $M$ be a closed orientable 4-manifold which is Seifert fibred over a 2-orbifold B, then
	
	(1) $B$ is spherical or bad if and only if $M$ admits geometry $\mathbb{S}^3\times\mathbb{E}$ or $\mathbb{S}^2\times\mathbb{E}^2$;
	
	(2) $B$ is flat if and only if $M$ admits one of geometries $\mathbb{E}^{4}$, $\mathbb{N}il^{4}$, $\mathbb{N}il^{3}\times \mathbb{E}$ or $\mathbb{S}ol^3\times\mathbb{E}$ except for two flat manifolds which are not Seifert \cite{Hi:2002};
	
	(3)  $B$ is hyperbolic if and only if $M$ admits one of geometries $\widetilde{\mathbb{S}L_2}\times \mathbb{E}$ and $ \mathbb{H}^{2}\times \mathbb{E}^{2}$, or $M$ is non-geometric.
\end{theorem}

\begin{theorem}\cite{Ue:1990}
	Let $M$, $N$ be two closed orientable Seifert 4-manifolds over aspherical bases with isomorphic fundamental groups, then $M$ is diffeomorphic to $N$. Moreover, if the base orbifolds of $M$ and $N$ are hyperbolic or $M$ and $N$ admit geometry  $\mathbb{N}il^{4}$ or $\mathbb{S}ol^3\times\mathbb{E}$, then the diffeomorphism is fiber-preserved between $M$ and $N$.
\end{theorem}
However, it needs to be clarified that the Seifert fibration structures of a Seifert 4-manifold may not be unique when the base orbifold is flat, spherical or bad, and the examples with geometries $\mathbb{E}^{4}$ and $\mathbb{N}il^{3}\times \mathbb{E}$ are exhibited in \cite{Ue:1988}. There also exist 4-dimensional Seifert manifolds over hyperbolic base such that they can not admit any 4-dimensional geometries. From the view of representing the Seifert fibred 4-manifolds with Seifert invariants, we use Theorem B \cite{Ue:1991} when the base orbifold is orientable:
\begin{theorem}
	Suppose $M$ is a 4-dimensional Seifert manifold over hyperbolic 2-orbifold without reflect circle. Then the Seifert fibered structure satisfies exactly one of following three conditions.
	
(1) $M$ admits $X=\mathbb{H}^2\times\mathbb{E}^2$-geometry where all monodromy matrices are powers of one common matrix $Q$ which is conjugate to $\begin{pmatrix}
		0&1\\-1&0
	\end{pmatrix}$ or $\begin{pmatrix}
		1&1\\-1&0
	\end{pmatrix}$ in $\SL_2(\mathbb{Z})$. Additionally, if all the monodromy are trivial, then the rational Euler number $e=(0,0)\in\mathbb{R}^2$;

(2) $M$ admits $X=\widetilde{\SL_2}\times\mathbb{Z}$-geometry where all the monodromy are trivial while rational Euler number $e\neq(0,0)$.

(3) $M$ could not admit any geometric structure in the sense of Thurston.

\end{theorem}
It is rather strange at first glance that the Euler number defines only for the trivial-monodromy case. This condition is actually reasonable since the Euler number should reflect whether a finite cover of extension is split which is indeed what happens in 3-dimensional Seifert fibred spaces. While for the 4-dimensional case when the Euler number is combined with non-trivial action, the above property is not guaranteed. For example we suppose a Seifert manifold $M$ with geometry $\mathbb{H}^2\times\mathbb{E}^2$ such that the monodromies are all trivial except for one monodromy matrix is $\begin{pmatrix}
	0&1\\-1&0
\end{pmatrix}$ with index 4. We could formally define its Euler number as $e=(a-\Sigma\frac{a_i}{m_i},b-\Sigma\frac{b_i}{m_i})$. When we take a 4-fold cover $M'$ of $M$ such that the orbifold $B'$ of $M'$ is a 4-fold cover of $B$ with respect to the generator of monodromy. However the Euler number of $M$ is $(I+A+A^2+A^3)e$ which is actually trivial since $A$ has index $4$ and $I+A+A^2+A^3=0$.

On the other hand, the condition of $M$ being geometric could also be rephrased from the view of geometric group theory by Hillman, and we follow the definition and results in \cite{Hi:2002}.

\begin{theorem}\label{cor:hyperbolic}\cite[Corollary 9.2.1]{Hi:2002}
	Let $M$ be a smooth closed 4-manifold with fundamental group $\pi$, then $M$ is homotopy equivalent to a Seifert fibred space with general fibre $T^2$ or Klein bottle over a hyperbolic 2-orbifold if and only if $h(\sqrt{\pi})=2$, $[\pi:\sqrt{\pi}]=\infty$ and $\chi(M)=0$.
\end{theorem}

\section{Mapping class group of surface}
Since the Seifert invariants including monodromy and fiber type all depend on the choice of generators and framing, it is necessary to modify the representation. When the orbifold contains no singular points, $\pi_1^{\orb}(B)\cong\pi_1(\Sigma_g)$. Consider the surface with $S=\Sigma_{g,n}$ genus $g$ and $n$ connected component of boundary, its mapping class group $$\Mod(S,\partial S)=\pi_0(\Homeo^+(\Sigma_g)),$$ consists all the isotopy class of homeomorphisms which fixing boundary and is orientation-preserving, i.e. $\Mod(\Sigma_g)=\Homeo^+(S,\partial S)/\Homeo^0(S,\partial S)$, where $\Homeo^0(S,\partial S)$ is the connected component of $\Homeo^+(S,\partial S)$ including identity. It is more convenient to consider the homeomorphism of surface fixing $n$ points, which leads to the mapping class group $\Mod(\Sigma_{g,n})$. The generator set of mapping class group is closely related to Dehn twist, which we mainly refer to the insightful book \cite{Fa-Ma} by Farb-Margalit. The following theorem explains why we care about mapping class group.
\begin{theorem}[Dehn-Nielson-Baer]
	 The homomorphism $$\sigma: \Mod^\pm(S_g)\to \Out(\pi_1(S_g))$$ is an isomorphism when $g\geq 1$.
\end{theorem}
Hense we could see $\Mod(S_g)$ as an subgroup of $\Out(\pi_1(S_g))$ of index 2. The annulus $A=S^1\times I$ with its mapping class group $\Mod(A)$ plays vital part when we consider the mapping class group of surface. Firstly we have:
\begin{theorem}
	The mapping class group of annulus $\Mod(A)\cong\mathbb{Z}$.
\end{theorem}

Every element in $\Mod(A)$ is associated to $\begin{pmatrix}
	1&n\\0&1
\end{pmatrix}$ with $n\in\mathbb{Z}$, deciding a linear transformation of the universal covering space $\tilde{A}=\mathbb{R}\times I$ which keeps the boundary of annulus after projection. when $n=-1$, the element $\begin{pmatrix}
	1&-1\\0&1
\end{pmatrix}$ decides a generator $\phi$ in $\Mod(A)$, which is the so-called Dehn twist.

\begin{figure}[htbp]
	\centering
	\includegraphics[scale=0.9]{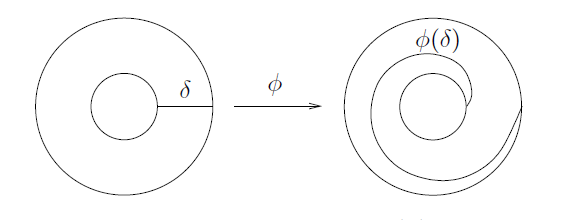}
	\caption{Dehn twist on the annulus \cite{Fa-Ma}}
\end{figure}

We could generalize Dehn twist on surface taking annulus as submanifold of surface and understand each element of $Mod(\Sigma_g)$ by its action on simple closed curves. The following vital result by Dehn.

\begin{theorem}
	The mapping class group $\Mod(T^2)$ acts on $H_1(T^2;\mathbb{Z})\cong\mathbb{Z}^2$, which induces an isomorphism $$\Mod(T^2)\to \SL(2,\mathbb{Z}).$$
\end{theorem}

Suppose $T:S^1\times I\to S^1\times I,\ (\theta,t)\mapsto (\theta-2\pi t,t)$ is the  homeomorphism on the annulus which generating $Mod(A)$. For a closed surface $S$ and a simple closed curve $\alpha$. Take the tubular neighborhood $N$ of $\alpha$ in $S$ and an orientation-preserving homeomorphism $\phi:A\to N$, we could construct an homeomorphism $T_\alpha:S\to S$ such that
\[	T_\alpha(x)=\left\{
\begin{array}{ll}
	\phi\circ T\circ\phi^{-1},& if \ x\in N,\\
	x,& if\  x\notin N.
\end{array}
\right.
\]

The construction  of $T_\alpha$ relies on the choice of $N$ and $\phi$, while its isotopy class is independent of $N$ by the uniqueness of $N$ up to homeomorphism. It is also independent of representation of the isotopy class of $\alpha$. After we use $a$ to denote the isotopy class of $\alpha$, it is well-defined to see $T_a$ as an element of $\Mod(S)$, which is called as the Dehn twist along $a$. It is topologically equivalent to cutting $S$ along $\alpha$, rotating $2\pi$ of the tubular neighborhood of one boundary and reglue.
	
\begin{figure}[htbp]
	\centering
	\includegraphics[scale=0.9]{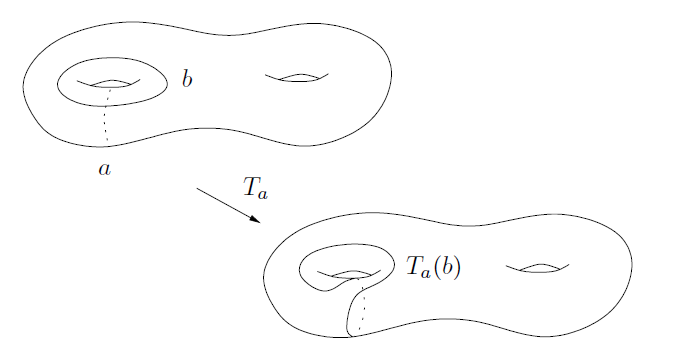}
	\caption{Cutting along $a$, twist and reglue of $T_a$ \cite{Fa-Ma}}
\end{figure}

If $b$ and $a$ are non-intersecting simple closed curves in $S$, i.e. $i(a,b)=0$, then $T_a(b)=b$; if $i(a,b)>1$, there is also an explicit way to see $T_a$ \cite{Fa-Ma}; If $a$ is not isotopic to one point or any boundary component, then $T_a\in \Mod(S)$ is non-trivial as long as $a$ is not isotopic to one point or to some component of boundary. 

The following theorem shows about the generating set of $\Mod(S)$.
\begin{theorem}[Dehn-Lickorish]
	$\Mod(S_g)$ is generated by finitely many Dehn twists.
\end{theorem}

It is firstly proved by Dehn in 1937 that $\Mod(S_g)$ is generated by $2g(g-1)$ Dehn twists\cite{Dehn}. Then Lickorish independently proved that $\Mod(S_g)$ could be generated by $3g-1$ Dehn twists, which is illustrated in the following picture \cite{Lick}. In 1979, Humphiries proves that $\Mod(S_g)$ could be generated by $2g+1$ Dehn twist, and the least number of Dehn twists of one generating set of $\Mod(S_g)$ is indeed $2g+1$.

\begin{figure}[htbp]
	\centering
	\includegraphics[scale=0.9]{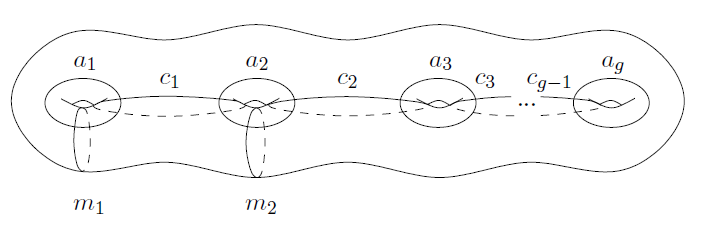}
	\caption{$3g-1$ Dehn twist generating$S_g$ \cite{Fa-Ma}}
\end{figure}

We are more concerning about the $\Mod(S_g)$-action on the first homology of surface. Suppose  $\phi\in \Homeo^+(S_g)$ inducing an isomorphism $\phi_*:H_1(S_2g;\mathbb{Z})\to H_1(S_2g;\mathbb{Z})$. If $\phi\thicksim\psi$ is equivalent in $\Mod(S_g)$, then it is natural that $\phi_*=\psi_*$. Hence each element in $\Mod(S_g)$ acts on $H_1(S_g;\mathbb{Z})$, which means that there is a representation $\Phi:\Mod(S_g)\to \Aut(H_1(S_g;\mathbb{Z}))$. We also consider the natural symplectic structure over $H_1(S_g;\mathbb{R})$ such that $\{x_1,y_1,\cdots,x_n,y_n\}$ forms a set of basis of $\mathbb{R}^{2g}$. Then the standard symplectic form on the dual space $(\mathbb{R}^{2g})^\star$ is the 2-form
$\omega=\sum_{i=1}^{g}dx_i\wedge dy_i.$
For any two vectors $v=(v_1,w_1,\cdots,v_g,w_g)$ and $v'=(v_1',w_1',\cdots,v_g',w_g')$ in $\mathbb{R}^{2g}$, we could compute
$$w(v,v')=\sum_{i=1}^{g}(v_iw_i'-v_i'w_i).$$
Then $w$ is the unique non-degenerated alternating bi-linear form after fixing the basis. The linear symplectic group $\Sp(2g,\mathbb{R})$ is the linear transformation of $\mathbb{R}^{2g}$ which preserves $w$. This definition could also be expressed in the matrix way.
$$\Sp(2g,\mathbb{R})=\{A\in \GL(2g,\mathbb{R}):A^\star\omega=\omega\}=\{A\in \GL(2g,\mathbb{R}):A^TJA=J)\},$$
where $J$ is the $2g\times 2g$ integer matrix
$$J=\begin{pmatrix}
	0&1&0&0&\cdots&0&0\\
	-1&0&0&0&\cdots&0&0\\
	
	0&0&0&1&\cdots&0&0\\
	0&0&-1&0&\cdots&0&0\\
	\vdots&\vdots&\vdots&\vdots&\ddots&\vdots \\
	
	0&0&0&0&\cdots&0&1\\
	0&0&0&0&\cdots&-1&0\\
	
\end{pmatrix}.$$
Then we could define the integer symplectic matrix $\Sp(2g,\mathbb{Z})=\Sp(2g,\mathbb{R})\cap \GL(2g,\mathbb{Z})$ whose matrix clearly has determinant $1$.

Now we go back to the representation $\Phi:\Mod(S_g)\to \Aut(H_1(S_g;\mathbb{Z}))\cong \GL(2g,\mathbb{Z})$ such that every element $f\in \Mod(S_g)$ preserves the lattice $H_1(S_g;\mathbb{Z})$ in $H_1(S_g;\mathbb{R})$, hence $\Phi(\Mod(S_g))\subset \SL(2g,\mathbb{Z})$. On the other hand, $f$ also preserves the 2-form over $H_1(S_g;\mathbb{R})$ which is defined by the algebraic intersection number of simple closed curves of surface. Hence $\Phi(\Mod(S_g))\subset \Sp(2g,\mathbb{R})$, which means that what we actually concern is the symplectic representation 
$$\Phi:\Mod(S_g)\to \Sp(2g,\mathbb{Z}).$$

The fundamental symplectic matrices of $\Sp(2g,\mathbb{Z})$ are defined as followed. Suppose $\sigma $ is an alternation of $\{1,2,\cdots,2g\}$ which permutes every tuple of $2i$ and $2i-1$ for $i=1,2,\cdots,g$. Then the fundamental symplectic matrix $${\rm SE}_{ij}=\left\{
\begin{array}{cc}
	I_{2g}+e_{ij} & if\ i=\sigma(j),\\
	I_{2g}+e_{ij}-(-1)^{i+j}e_{\sigma(j)\sigma(i)} & if\ i\neq\sigma(j),
\end{array}
\right.
$$
The matrix $e_{ij}$ is the matrix that the $(i,j)$-entry equal to 1 and any others are 0 when $i\neq j$. Then following theorem gives a generating set of integer symplectic matrices.
\begin{theorem}[\cite{Sp-generator}]
	$\Sp(2g,\mathbb{Z})$ is generated by all the fundamental symplectic matrices $\{{\rm SE}_{ij}\}$. 
\end{theorem}

Burkhardt gave another set of generators of $\Sp(4,\mathbb{Z})$ early in 1890. Suppose $\mathbb{Z}^4$ is the vector space spanned by $x_1,y_1,x_2,y_2$, then $\Sp(4,\mathbb{Z})$ could be generated by following four types of transformations.

$(x_1,y_1,x_2,y_2)\mapsto(x_1+y_1,y_1,x_2,y_2)$,

$(x_1,y_1,x_2,y_2)\mapsto(y_1,-x_1,x_2,y_2)$,

$(x_1,y_1,x_2,y_2)\mapsto(x_1-y_2,y_1,x_2-y_1,y_2)$,

$(x_1,y_1,x_2,y_2)\mapsto(x_2,y_2,x_1,y_1)$.

When $g>2$, we could add transformations which are similar to the last type which forms a generating  set $\Sp(2g,\mathbb{Z})$.

\section{Equivalence of monodromies for $\mathbb{H}^2\times\mathbb{E}^2$-manifolds}

After the building of connection of 4-dimensional geometries and profinite completion involved some geometric group property, we could consider the profinite rigidity of Seifert 4-manifolds. Actually, the counterexamples have been raised based on the work of Funar\cite{Funar} and Hempel\cite{Hem:2014} in whose fundamental group contains $\mathbb{Z}$ as direct factor. But we still care about the Seifert 4-manifolds over hyperbolic 2-orbifolds whose fundamental groups contain the unique $\mathbb{Z}^2$ normal subgroup which is preserved under the isomorphism between profinite completions \cite{MaWa:2022}. The efficient expression of the  4-dimensional Seifert invariants defined by Ue depends on the choice of generators of one's fundamental group, including the action of orbifold genus on fibers which is expressed as monodromy matrices. We compute the equivalence of monodromy of $\mathbb{H}^2\times\mathbb{E}^2$-manifolds and give Theorem \ref{monodromy} such that  there is always a set of generators of $\pi_1(M)$ to make sure the corresponding monodromies are all trivial except for one matrix.

\begin{proof}[\textbf{Proof of Theorem {\rm \ref{monodromy}}. }]
	By assumption that $M$ admits Seifert 4-structure over hyperbolic $2$-orbifold, and we could  suppose this $2$-orbifold is surface $\Sigma_g$ since the cone points don't affect genus. The non-trivial monodromy means there is an associating non-trivial $\mathbb{Z}[\pi_1(\Sigma_g)]$-module structure on $\pi_1(T^2)\cong\mathbb{Z}^2$ which is the fundamental group of the fiber. Then monodromy matrices $A_i,B_i$ are all powers of a common matrix $Q$, which is periodic and conjugate to $\begin{pmatrix}
		0&1\\-1&0
	\end{pmatrix}$ or $\begin{pmatrix}
		1&1\\-1&0
	\end{pmatrix}$ in $\SL_2(\mathbb{Z})$. Fixing a set of generators $\{x_1,y_1,\cdots,x_g,y_g\}$ of $\pi_1(\Sigma_g)$, we suppose the matrices are all first type such that $A_i=Q^{a_i}$ and $B_i=Q^{b_i}$ are powers of $Q$ with order $4$. Then there is a natural correspondence between monodromy matrices and cohomology groups.
	$$\begin{array}{ccc}
		\{\mbox{The monodromy matrices of } M\}&\longrightarrow & H_1(\Sigma_g;\mathbb{Z}/4\mathbb{Z})\cong(\mathbb{Z}/4\mathbb{Z})^{2g}\\
		(A_1,B_1,\cdots,A_g,B_g)&\longmapsto&(a_1,b_1,\cdots,a_g,b_g)
	\end{array}$$

	By taking another set of generators of $\Sigma_g$ is the existence of $\phi\in \Mod(\Sigma_g)$ such that $\phi(x_1),\phi(y_1),\cdots,\phi(x_g),\phi(y_g)$ generates $\pi_1(\Sigma(g))$. Then the associated powers of their monodormies $(a'_1,\cdots,b'_g)$ and $(a_1,\cdots,b_g)$ satisfy that there exists a sympletic matrix $P\in \Sp(2g,\mathbb{Z}/4\mathbb{Z})$ such that
	$$(a'_1,\cdots,b'_g)^T=P\cdot(a_1,\cdots,b_g)^T,$$
	which is indeed the equivalence condition of monodromy under different generators. Then the reclassfying of monodromies of  $M$ has been transfer to compute the orbit of $H_1(\Sigma_g;\mathbb{Z}/4\mathbb{Z})$ under the action of $\Sp(2g,\mathbb{Z}/4\mathbb{Z})$. Since the finite cyclic coefficient group of cohomology of surface group, the first and second type of transformation given by Burkhardt guarantee that any tuple  $(a_1,b_1,\cdots,a_g,b_g)\in(\mathbb{Z}/4\mathbb{Z})^{2g}$could always been transfer to $(a'_1,0,\cdots,a'_g,0)$  to make all the $b_i=0$, i.e. there exists $P_1\in \Sp(2g,\mathbb{Z}/4\mathbb{Z})$ such that $P_1\cdot(a_1,b_1,\cdots,a_g,b_g)^T=(a'_1,0,\cdots,a'_g,0)$. Then again the 3rd and 4th type of transformation make every entry equals to $0$ except for the first one, i.e. there exists $P_2\in \Sp(2g,\mathbb{Z}/4\mathbb{Z})$ such that $P_2\cdot(a'_1,0,\cdots,a'_g,0)^T=(a''_1,0,\cdots,0,0).$
\end{proof}

This theorem shows us that the possible monodromy falls into a very limited range that the non-trivial  $\mathbb{Z}[\pi_1(\Sigma_g)]$-module  $\pi_1(T^2)\cong\mathbb{Z}^2$ is quite strict. 

\section{Profinite completion and Seifert 4-manifolds}
We now jump to the situation when all the monodromy matrices are trivial and $M$ admits $\widetilde{\SL_2}\times\mathbb{E}$ geometry.
The rational Euler number is defined exactly in this situation which is also relied on the framing of fiber. We prove that the profinite completion determines the rational Euler number up to the natural $\SL(2,\mathbb{Z})$-action.

\begin{theorem}\label{euler}
	Suppose $M$ and $N$ are two Seifert 4-manifolds with  $\widetilde{\SL_2}\times\mathbb{E}$ geometry and isomorphic profinite completions. Then there exists sets of generators of $\pi_1(M)$ and $\pi_1(N)$ separately such that the accordingly $e(M)=e(N)=(0,c)\in\mathbb{Q}^2$.
\end{theorem}

\break

\begin{proof}
	
	We denote the base orbifolds of $M$ and $N$ as $B_1$ and $B_2$ respectively. It is already proved that the isomorphism between profinite completions  $\phi:\widehat{{\pi_1(M)}}\to\widehat{{\pi_1(N)}}$ send $\widehat{{\sqrt{\pi_1(M)}}}$ to $\widehat{{\sqrt{\pi_1(N)}}}$ since the nilradical $\sqrt{\pi_1(M)}$ is the unique maximal normal $\widehat{ \mathbb{Z}^2}$-subgroup in $\pi_1(M)$ which is also holds in the profinite analogue. Then $\pi_1(M)$ is actually an extension of $\pi_1^{\orb}(B_1)$ by $\mathbb{Z}^2$ and  $\phi$ induces the isomorphism between the nilradicals $\phi:\widehat{\sqrt{\pi_1(M)}}\to\widehat{\sqrt{\pi_1(N)}}$ along with the isomorphism between their quotients $\bar{\phi}:\widehat{\pi_1^{\orb}(B_1)}\to\widehat{\pi_1^{\orb}(B_2)}$.
	\centerline{\xymatrix{
			1\ar[r]&\widehat{\mathbb{Z}^2} \ar[r]\ar[d]_{\phi} & \widehat{\pi_1(M)}\ar[r]\ar[d]_{\phi}&\widehat{\pi_1^{\orb}(B_1)}\ar[r]\ar[d]_{\bar{\phi}}&1 \\
			1\ar[r]&\widehat{\mathbb{Z}^2} \ar[r] & \widehat{\pi_1(N)}\ar[r]&\widehat{\pi_1^{\orb}(B_2)}\ar[r]&1,
	}}
	\noindent The Fuchsian groups are profinitely rigid among themselves based on the work of \cite{BCR:2015}, which actually shows that $\bar{\phi}:\widehat{\pi_1^{\orb}(B_1)}\to\widehat{\pi_1^{\orb}(B_2)}$ induces the homeomorphism $B_1\cong B_2$ since they are all aspherical.
	
	The group presentation of $\pi_1(M)$ has following form after fixing generators,
	\begin{align*}
		\pi_1(M)=\langle u_1,v_1,\cdots,u_g,v_g,x_1,\cdots,x_r,l,h|\ x_1\cdots x_r\prod[u_1,v_1]\cdots[u_g,v_g]=l^ah^b,& \\x_i^{m_i}l^{a_i}h^{b_i}=1,\ <l,h>\in Z(\pi_1(M))&\rangle
	\end{align*}
	where $u_1,\cdots,v_g$ are corresponding the genus in $B_1$, and $x_1,\cdots,x_r$ are related to the cone points in $B$ whose fiber type is described by Seifert invariants $(m_i,a_i,b_i)$. If we choose different framing of fiber as another set of generators of $\sqrt{\pi_1(M)}$ such that $\sqrt{\pi_1(M)}=\langle l',h'\rangle$ and $(l',h')=(l,h)P$ with $P\in \SL(2,\mathbb{Z})$. Then the relative Seifert invariants satisfy $(a'_i,b'_i)^T=P\cdot(a_i,b_i)^T$ by the basic linear transformations. Since the definition of rational Euler number only relies on Seifert invariants, it is easily to see that $$e'(M)=(a'-\sum\frac{a'_i}{m_i},b'-\sum\frac{b'_i}{m_i})=e(M)\cdot P^T.$$
	Then we can always find an appropriate matrix $P\in \SL(2,\mathbb{Z})$ to control that there is only one entry in $e(M)$ is non-trivial under the framing of $(l,h)\cdot P^T$ because of the rationality of $e(M)$, hence we have
	$$e(M)\equiv (0,b)\in \mathbb{Q}^2\ \modd\  \SL(2,\mathbb{Z}).$$
	
	On the other hand, suppose $M'$ is the covering space of $M$ with smallest index $d$ such that $M'$ is Seifert 4-manifold over some surface. Then the base orbifold $B'$ of $M'$ is actually a surface cover of $B$ with index $d$. We choose the framing of $\sqrt{\pi_1(M')}$ which is coherent with the framing in $\sqrt{\pi_1(M)}$, then we have
	
	$$e(M')=d\cdot e(M)=(0,z)\in \mathbb{Z}^2.$$
	
	In the profinite world, $\overline{\pi_1(M')}$ is a $d$-index closed normal subgroup of $\widehat{{\pi_1(M)}}$. There must exists a $d$-index normal subgroup $\pi_1(N')$ of $\pi_1(N)$ such that $\widehat{{\pi_1(M')}}\cong\widehat{{\pi_1(N')}}$. By the profinite invariance of first betti number $b_1$, there is always a set of generators of $\pi_1(N)$ such that the correspondingly 
	Seifert invariants satisfying $e(N')=e(M')=(0,z)\in\mathbb{Z}^2$, which proves that $$e(M)=\frac{1}{d}\cdot e(M')=\frac{1}{d}\cdot e(N')=e(N)\in\mathbb{Q}^2.$$
\end{proof}

When all the monodromies are trivial, the isomorphism between profinite group induces an action on the cohomology of orbifold fundamental group. Suppose $M$ a $\widetilde{\SL_2}\times\mathbb{E}$-manifold or a $\mathbb{H}^2\times\mathbb{E}^2$-manifold with trivial monodromy over hyperbolic 2-orbifold $B$, then its fundamental group suits a central extension of $\mathbb{Z}^2$ by $\pi_1^{\orb}(B)$. The following analysis is exactly Seifert 4-dimensional version of Wilkes's about 3-dimensional Seifert fibred spaces \cite{Wilk:2017} and in classic group cohomology theory. A central extension of group $Q$ by group $N$ satisfies following short exact sequence
$$1\to N\to G\to Q\to 1,$$
such that the image of $N$ in $G$ lies in the center of $G$. We say two extension $G$ and $G'$ which are central extension of $Q$ by $N$ are equivalent as extension, if they are commuted in following diagram

\centerline{\xymatrix{
		1\ar[r]&N \ar[r]\ar[d]_{Id} & G\ar[r]\ar[d]_{\cong}&Q\ar[r]\ar[d]_{Id}&1 \\
		1\ar[r]&N\ar[r] & G\ar[r]&Q\ar[r]&1
}}
The equivalence class of extension of $N$ by $Q$ is one-to-one correspondent to the elements in $H^2(Q,N)$ from the basic group homology. Then suppose the group presentation of $Q=\langle x_1,\cdots,x_n|r_1,\cdots,r_m\rangle$, and $F=\langle x_1,\cdots,x_n\rangle$ is the free group generated by same $n$ generators, while $R$ is the normal closure of all the relations $\{r_1,\cdots,r_m\}$, which leads to the short exact sequence 
$$1\to R\to F\to Q\to 1,$$
and importantly there is the long exact sequence induced by Serre spectrum sequence
$$0\to H^1(Q;N)\to H^1(F;N)\to (H^1(R;N))^F\to H^2(Q;N)\to H^2(F;N)$$
where $(H^1(R;N))^F$ is the elements in $H^1(R;N)$ which is invariant under the conjugate action of $F$. $$(H^1(R;N))^F\cong H^1(R/[R,F];N)$$
The last term $H^2(F;N)=0$ since $F$ is free, hence the homomorphism $(H^1(R;N))^F\to H^2(Q;N)$ is surjective. For any extension $\xi\in H^2(Q;N)$, we could always consider its preimage $\zeta\in H^1(R/[R,F];N)$ which has more clear way to express. The automorphism groups of $Q$ and $N$ separately induce left/right action on $H^2(Q;N)$, which leads to the core of our question that when would two central extension are isomorphic as groups allowing non-trivial isomorphisms of $Q$ and $N$, and whether these condition maintains under profinite completion and passing to the Seifert invariants.

The group $Q=\pi_1^{\orb}(B)$ is the orbifold fundamental group and $N\cong \mathbb{Z}^2$, then we consider the central extension
$$1\to \mathbb{Z}^2\to \pi_1(M)\to Q\to 1$$ and the correspondence with $H^2(Q;\mathbb{Z}^2)$. By the goodness of 2-orbifold group, the cohomology of $Q$ and its profinite completions are isomorphic when the coefficient is finite $\mathbb{Z}[Q]$-module. Luckily the isomorphism between profinite completions $\widehat{{\pi_1(M)}}\cong\widehat{{\pi_1(N)}}$ preserves nilradicals which are the profinite completions of fibers. On the other hand, $\mathbb{Z}^2$ has unique $t^2$-index subgroup such that the quotient is isomorphic to $(\mathbb{Z}/t\mathbb{Z})^2$ for every $t\in \mathbb{Z}$. Then $\widehat{{\pi_1(M)}}\cong\widehat{{\pi_1(N)}}$ also induces the isomorphism between these $t^2$-index normal subgroup $$G=\widehat{{\pi_1(M)}}/\overline{\langle l^t,h^t\rangle}\cong\widehat{{\pi_1(N)}}/\overline{\langle l'^t,h'^t\rangle}\cong G'$$    
where $G,G'$ is the central extension of $\widehat{Q}$ by $(\mathbb{Z}/t\mathbb{Z})^2$, which is associating to the cohomology elements $\xi_{t},\xi'_{t}\in H^2(\widehat{Q},(\mathbb{Z}/t\mathbb{Z})^2)\cong H^2(Q,(\mathbb{Z}/t\mathbb{Z})^2)$. The action on cohomology induced by automorphisms of $\widehat{Q}$ and $(\mathbb{Z}/t\mathbb{Z})^2$ sends $\xi_{t}$ to $\xi'_{t}$. Also, as two extensions  of $Q$ by $\mathbb{Z}^2$, $\pi_1(M)$and $\pi_1(N)$ are corresponding to two elements $\xi,\xi'\in H^2(Q,\mathbb{Z}^2)$ with following surjection when considering $\xi_{t}$ and $\xi'_{t}$ as elements in $H^2(Q,(\mathbb{Z}/t\mathbb{Z})^2)$.
$$\begin{array}{ccc}
	H^2(Q,\mathbb{Z}^2)&\longrightarrow & H^2(Q,(\mathbb{Z}/t\mathbb{Z})^2)\\
	\xi &\longmapsto&\xi_{t}
\end{array}$$

The group presentation of $\pi_1(M)$ is given in the last section, which is also available for $\mathbb{H}^2\times\mathbb{E}^2$-manifold with trivial monodromy. The orbifold fundamental group of base orbifold $B$ is $Q=\pi_1^{\orb}(B)=\langle u_1,\cdots,v_g,x_1,\cdots,x_r|x_i^{m_i}=1,x_1\cdots x_r\prod[u_i,v_i]=1\rangle$. Then the short exact sequence $1\to R\to F\to Q\to 1$ corresponds the generators and generating relations of $Q$. Actually $R/[R,F]$ is the free $\mathbb{Z}$-module spanned by $y_0=x_1\cdots v_g^{-1}$ and $y_i=x_i^{m_i}$ such that $\pi_1(M)$ the cohomology element $H^2(Q,\mathbb{Z}^2)$ representing $\pi_1(M)$ has a preimage in $(H^1(R;\mathbb{Z}^2))^F\cong \Hom(R/[R,F],A)$ as 
$$y_0\mapsto (a,b),\ \ y_i\mapsto (-a_i,-b_i).$$

It should be noted that this correspondence of Seifert 4-manifolds is also suitable for the non-trivial monodromies. Back to the cohomology of groups that an extension of group $G$ by an abelian kernel $A$ is associated to a function $f:G\times G\to A$ which could be proved as a 2-cocycle. 
$$0\to A\to^{i} E\to^{\pi} G\to 1$$ 
After giving a set-theoretical section $s:G\to E$ such that $\pi s=id_{G}$, the definition of $f$ describes the gap of $s$ being a homomorphism, i.e. the extension being split. The concrete definition of $f$ is given by 
$$s(g)s(h)=i(f(g,h))s(gh)$$ for any $g,h\in G$. Then for a classical cross-section $s:G\to E$ such that it maps each generator in the quotient $G$ to $E$ in the canonical way and sends $1_{G}$ to $1_{E}$. Then the definition of $f$ shows exactly that in our problem, a generating relation $y_i=x_{i}^{m_i}$ corresponding to $l^{a_i}h^{b_i}$ in $A=\mathbb{Z}^2$ and $y_0=x_1\cdots v_g^{-1}$ corresponding to $l^{a}h^{b}$ which is useful for further computation.

Now we exam the automorphism of $Q$ with following result to show the genre of specific condition.
\begin{theorem}[\cite{BCR:2015}]
	Suppose $G$ is a finitely generated Fuchsian group. Then every finite subgroup of $\widehat{G}$ is  conjugate some finite subgroup of $G$, what's more, if two maximal finite subgroup of $G$ is conjugate in $\widehat{G}$, then they are already conjugate in $G$.
\end{theorem}

Then following theorem depicts the automorphism of $\widehat{Q}$.

\begin{proposition}[\cite{Wilk:2017}]
	Suppose $B$ is a hyperbolic 2-orbifold with 
	$$\pi_1^{\orb}(B)=\langle x_1,\cdots,x_r,u_1,\cdots,v_g|x_i^{m_i}=1,x_1\cdots x_r\prod[u_i,v_i]=1\rangle, $$ then the automorphism of its profinite completion $\widehat{\pi_1^{\orb}(B)}$ sends $a_i$ to some conjugation of $a_i^{k_i}$ such that $k_i$ is some integer coprime to $m_i$.
\end{proposition}

Finally we deal with the action induced on 
$H^2(\widehat{Q},(\mathbb{Z}/t\mathbb{Z})^2)$ and how this action maintains in profinite completions.
\begin{proposition}
	Suppose $\phi$ is an isomorphism of  $\widehat{Q}$ as above. Then  $\phi^*$ induced by $\phi$ has an action on $H^2(\widehat{Q},(\mathbb{Z}/t\mathbb{Z})^2)$ which is actually multiplying a profinite integer $\kappa\in\mathbb{Z}$ such that $1\leq i\leq \kappa$ for any $\kappa\equiv k_i\ \modd\ m_i$.
\end{proposition}

\begin{proof}
	This proposition is largely depended on the 
	Suppose $Q=\pi_1^{\orb}(B)$ also has group presentation as above. Then by the work of Wilkes \cite{Wilk:2017}, there is a projective resolution of $\mathbb{Z}$ by free $\mathbb{Z}[Q]$-module $$C_3\to C_2\to C_1\to C_0\to\mathbb{Z}.$$ This whole construction is based on the Cayley graph of $Q$ with its natural $\mathbb{Z}[Q]$-action. 
	
	Suppose $F=F(x_i,u_i,v_i)$ is the free group generated by same generators of $Q$ and $R=\Ker(F\to Q)$, then a potential projective resolution is constructed by setting $C_0=\mathbb{Z}Q$ which is a free $\mathbb{Z}$-module  generated by all the points in $\Cay(Q)$. The left translation of $\Cay(Q)$ acted by $Q$ also lifts to $C_0$. And define $\epsilon:\mathbb{Z}Q\to\mathbb{Z}$ is the evaluation map.
	
	Let $C_1=\mathbb{Z}Q\{\bar{x}_i,\bar{u}_j,\bar{v}_j\}$ generated by all the edges in $\Cay(Q)$ with start point $1$. Then $C_1$ is a free $\mathbb{Z}[Q]$-module generated by all $\bar{x}_i,\bar{u}_j,\bar{v}_j$ which is actually representing all the linear combinations of roads in $\Cay(Q)$. Then the partial map sends every road to its endpoint minus start point. It could be easily examined that $\epsilon\dot d_1=0$ and $\Ker(\epsilon)=Im(d_1)$.

	Let $C_2=\mathbb{Z}Q\{y_0,\cdots,y_r\}$ generated by all the relations $y_0=x_1\cdots x_r\prod [u_i,v_i], y_1=x_1^{m_1}\cdots,y_r=x_r^{m_r}$. Then $C_2$ could be seen as all the loops in $\Cay(Q)$ with $Q$-action as left translation. The partial map $d_2:C_2\to C_1$ sends every generator to the sum of all the edges consisting the loop such as 
	\begin{align*}
		d_2(y_i)=&\bar{x}_i+x_i\cdot \bar{x}_i+\cdots+x_i^{m_i-1}\cdot \bar{x}_i\\
		d_2(y_0)=&\bar{x}_1+x_1\cdot \bar{x}_2+x_1x_2\cdot \bar{x}_3+\cdots+x_1\ldots x_{r-1}\cdot \bar{x}_i\\
		&+x_1\ldots x_r\cdot\bar{u}_1+\cdots-x_1\ldots x_r[u_1,v_1]\cdots[u_g,v_g]\bar{v}_g.
	\end{align*}
which also keeps $ker(d_1)=Im(d_2)$.
	
	By some computation, we could construct $C_3=\mathbb{Z}Q\{z_1,\cdots,z_r\}$ such that $z_i$ is corresponding to the relation $y_i$ and  $d_3(z_i)=(x_i-1)\cdot y_i$ with $ker(d_2)=Im(d_3)$, forming a projective resolution of $\mathbb{Z}$ by $\mathbb{Z}Q$-modules.
	
	The profinite modules $\widehat{C_i}$ form a projective resolution of $\widehat{\mathbb{Z}}$ by $\widehat{\mathbb{Z}}[[\widehat{Q}]]$-module with profinite generators corresponding to discrete ones by the work of \cite{Wilk:2017}. Then $\phi$ induces chain maps $\phi_\sharp:\widehat{C_i}\to\widehat{C_i}$ and also homomorphisms $\phi^*$ of cohomology groups.

	To make following diagram commute,
	
	\centerline{\xymatrix{
			\widehat{C_3}\ar[r]^{d_3}&\widehat{C_2} \ar[r]^{d_2}\ar[d]_{\phi_\sharp} & \widehat{C_1}\ar[r]\ar[d]_{\phi_\sharp}&\widehat{C_0} \\
			\widehat{C_3}\ar[r]^{d_3}&\widehat{C_2}\ar[r]^{d_2} & \widehat{C_1}\ar[r]&\widehat{C_0}
	}}
	\noindent We would choose $\phi_\sharp(y_i)=k_ig_i\cdot y_i$ for some $g_i\in\widehat{Q}$ when $1\leq i\leq r$ since $\phi$ sends $x_i$ to some conjugation of $x_i^{k_i}$ in $\widehat{Q}$. 
	
	After applying functor $\widehat{\mathbb{Z}}\otimes_{\widehat{\mathbb{Z}}[[\widehat{Q}]]}-$ on $\widehat{C_i}\to\widehat{\mathbb{Z}}$ to compute $H^2(\Hom_{\widehat{\mathbb{Z}}[[\widehat{Q}]]}(\widehat{C_i},(\mathbb{Z}/t\mathbb{Z})^2))$, we get following commute diagram
	
	\centerline{\xymatrix{
			\widehat{\mathbb{Z}}\otimes_{\widehat{\mathbb{Z}}[[\widehat{Q}]]}\widehat{C_3}\ar[r]^{0}&\widehat{\mathbb{Z}}\{y_0,\cdots,y_r\} \ar[r]^{d'_2}\ar[d]_{\phi_\sharp} & \widehat{\mathbb{Z}}\{\bar{x}_i,\bar{u}_j,\bar{v}_j\}\ar[r]^{0}\ar[d]_{\phi_\sharp}&\widehat{\mathbb{Z}}\otimes_{\widehat{\mathbb{Z}}[[\widehat{Q}]]}\widehat{C_0} \\
			\widehat{\mathbb{Z}}\otimes_{\widehat{\mathbb{Z}}[[\widehat{Q}]]}\widehat{C_3}\ar[r]^{0}&\widehat{\mathbb{Z}}\{\bar{x}_i,\bar{u}_j,\bar{v}_j\}\ar[r]^{d'_2} & \widehat{\mathbb{Z}}\{\bar{x}_i,\bar{u}_j,\bar{v}_j\}\ar[r]^{0}&\widehat{\mathbb{Z}}\otimes_{\widehat{\mathbb{Z}}[[\widehat{Q}]]}\widehat{C_0}
	}}
	\noindent with the diminish of $\widehat{Q}$-action and chain  maps such that 
	\begin{align*}
		d'_2(y_i)=m_ix_i,\ 
		d'_2(y_0)=x_1+\cdots+x_r,\ 
		\phi_\sharp(x_i)=k_i x_i,\ 
		\phi_\sharp(y_i)=k_i y_i.
	\end{align*}
	We still don't know what $\phi_\sharp(y_0)$ looks like, which could be assumed as  $\phi_\sharp(y_0)=\kappa y_0+\sum \mu_iy_i$ where $\kappa,\mu\in\mathbb{Z}$. Then by the commutation of above diagram, it holds for every $1\leq i\leq r$ such that $$\kappa+m_i\mu_i=k_i.$$
	
	Then we apply functor $\Hom_{\widehat{\mathbb{Z}}}(-,(\mathbb{Z}/t\mathbb{Z})^2)$- to the diagram for any $t\in\mathbb{N}$, we could get a new phrase of diagram
	
	\centerline{\xymatrix{
			&	(\mathbb{Z}/t\mathbb{Z})^2\{y^*_0,\cdots,y^*_r\}\ar[l] &
			(\mathbb{Z}/t\mathbb{Z})^2\{\bar{x}^*_i,\bar{u}^*_j,\bar{v}^*_j\}\ar[l]^{d^2} &\ \ar[l]\\
			\ &	(\mathbb{Z}/t\mathbb{Z})^2\{y^*_0,\cdots,y^*_r\}\ar[u]^{\phi^\sharp}\ar[l]&
			(\mathbb{Z}/t\mathbb{Z})^2\{\bar{x}^*_i,\bar{u}^*_j,\bar{v}^*_j\}\ar[l]^{d^2}\ar[u]^{\phi^\sharp} & \ \ar[l]
	}}
	The horizontal maps provide us to compute  $H^2(\widehat{Q},(\mathbb{Z}/t\mathbb{Z})^2)$ while the vertical ones induces the homomorphism between cohomology groups. Now the cochain maps are like
	\begin{align*}
		\phi^\sharp(y^*_0)&=\kappa y^*_0\\
		\phi^\sharp(y^*_i)&=\mu_i y^*_0+k_iy^*_i\\
		d^2(\bar{x}^*_i)&=y^*_0+m_iy^*_i\\
		d^2(\bar{u}^*_0)&=d^2(\bar{v}^*_0)=0
	\end{align*}
	Hence for any 2-cocycle $\xi=(a,b)y^*_0-\sum(a_i,b_i)y^*_i$ associated to one Seifert 4-manifold, we have 
	\begin{align*}
		\phi^*([\xi])=\left[\phi^\sharp(\xi)\right]=&[(\kappa(a,b)-\sum\mu_i(a_i,b_i))y^*_0-\sum k_i(a_i,b_i)y^*_i]\\
		=&[\kappa((a,b)y^*_0-\sum (a_i,b_i)y^*_i)-\sum \mu_i(a_i,b_i)(y^*_0+m_iy^*_i)]\\
		=&\kappa[\xi]
	\end{align*}
\end{proof}

Now it's sufficient to prove the condition for two  $\mathbb{H}^2\times\mathbb{E}^2$-manifolds to have isomorphic  profinite completions if and only if their Seifert invariants are $(m_i,a_i,b_i)$ and $(m_i,ka_i,kb_i)$ where $k\in\mathbb{Z}$ is some integer relevant to above profinite integer.

\begin{proof}[\textbf{Proof of Theorem {\rm \ref{Seifert}}}. ]
	By the discussion above, the fundamental groups of  $M$ and $N$ are corresponding to 2-cocycles $\xi,\xi'\in H^2(Q,\mathbb{Z}^2)$. Suppose the Seifert invariants of $M$ are $((m_i,a_i,b_i),(a,b))$, then the lift of $\xi$ in $\Hom_{\mathbb{Z}B}(C_2,\mathbb{Z}^2)$ is 
	$$\xi=[(a,b)y^*_0-\sum (a_i,b_i)y^*_i],$$
	which is similar to $N$. And $\xi,\xi'$ correspond to $\xi_{i,n}\in H^2(\widehat{Q},(\mathbb{Z}/t\mathbb{Z})^2)$. We use  $\phi:\widehat{\pi_1(M)}\to\widehat{\pi_1(N)}$ to denote the isomorphism, then there is a series of surjections $\Aut(\widehat{\mathbb{Z}}^2)\to \SL(2,\mathbb{Z}/t\mathbb{Z})$ such that each $\phi|_{\widehat{\mathbb{Z}}}$ is associated to a series of matrices $\{A_t\in \SL(2,\mathbb{Z}/t\mathbb{Z})\}$. Then $\phi|_{\widehat{\mathbb{Z}}}$ and $\bar{\phi}:\widehat{Q}\to\widehat{Q}$ induce the action on cohomology  such that there exists $\kappa\in\mathbb{Z}$ satisfying $A_t\cdot\kappa\xi_{t,1}=\xi_{t,2}$. 
	
	For the manifolds with $\widetilde{\SL_2}\times\mathbb{E}-$geometry, we could define a group homomorphism to connect 2-cocycle with Euler number. Let $s$ is largely enough such that $t=se\prod m_i$, then we define $E:H^2(\hat{Q},(\mathbb{Z}\slash t\mathbb{Z})^2)\to(\mathbb{Z}\slash t\mathbb{Z})^2$ such that $$E((p,q)y^*_0-\sum (p_i,q_i)y^*_i)=(p,q)\prod m_i-\Sigma(p_i,q_i)\prod_{j\neq i}m_j.$$ It is easy to see that $E(\kappa \xi)=\kappa E(\xi)$ and also commute with action of $\Aut(\widehat{\mathbb{Z}}^2)$. By the former discussion, we could always find generating sets of $\pi_1(M)$ and $\pi_1(N)$ to make sure $e(M)=e(N)=(0,c)\in\mathbb{R}$, and $E(\xi_{t,2})=(0,c) \prod m_i\mod t$. Then we make alternation about the generators $(l,h)$ of $\sqrt{\pi_1(M)}$ such that it is generated by $\bar{\phi}^{-1}\cdot(l,h)$. Now $E(A_t\cdot\kappa\xi_{t,1})=A_tA_t^{-1}\cdot\kappa (0,c)\cdot \prod m_i=(0,c)\cdot \prod m_i=E(\xi_{t,2}) \mod t$. Then $\kappa-1\equiv0 \mod s$  holds for any $s$ if and only if $\kappa=1$ leading to the profinite rigidity.

	Actually we could always choose generators of fiber in $\pi_1(N)$ to diminish $A_t$. Now $\phi|_{\widehat{\mathbb{Z}}}$ could be formed to identity since  $\Aut(\widehat{\mathbb{Z}}^2)$ is residually finite. Then the chain maps satisfy $\kappa\xi_{t,1}=\xi_{t,2}$ for any  $t\in\mathbb{N}$.Take $t=\prod m_i$, then $\kappa\xi_{t,1}=\xi_{t,2}$ in  $H^2(\widehat{Q},(\mathbb{Z}/t\mathbb{Z})^2)$. Let $k\in\mathbb{Z}$ is the integer such that $k\equiv \kappa\ \modd \prod m_i $. Then the Seifert invariant of $M_2$ have to be $(m_i,ka_i,kb_i)$.
\end{proof}

There are still many types of Seifert 4-manifolds not solved about profinite rigidity, since the monodromies of fiber are not trivial leading to the non-trivial module structure of $\pi_1-$fiber which is also the next period of our work.

\noindent\textbf{Data availability statement}

This manuscript has no associated data.

\bibliographystyle{amsplain}

\end{document}